\documentclass[a4paper,12pt]{article}
\usepackage{geometry,latexsym,amssymb}
\usepackage[latin5]{inputenc}
\usepackage{enumerate}
\usepackage{enumitem}
\usepackage[usenames]{color}
\usepackage{graphicx}
\usepackage{enumerate}
\usepackage{rotating}
\usepackage{multirow}
\usepackage{cite}
\usepackage{amsthm}
\usepackage{caption}
\usepackage{tikz}
\usetikzlibrary{matrix,arrows}
\usepackage{amsmath,amscd}
\usepackage{calc}
\usepackage{tabularx}
\usepackage{ifthen}
\usepackage{eucal}
\usepackage{amssymb,amscd,amsthm, verbatim,amsmath,color,fancyhdr, mathrsfs}
\usepackage{graphicx}
\usepackage{turnstile}
\usepackage[colorinlistoftodos]{todonotes}

\newtheorem{thm}{Theorem}[section]
\newtheorem{prop}[thm]{Proposition}
\newtheorem{lem}[thm]{Lemma}
\newtheorem{rem}[thm]{Remark}
\newtheorem{cor}[thm]{Corollary}

\newtheorem{ex}[thm]{Example}
\newtheorem{defn}[thm]{Definition} 
\renewcommand{\arraystretch}{0.0 }

\begin{document}

\begin{center}
{\Large \textbf{Integer Representations of the Generalized Symmetric Groups}} \vspace*{0.5cm}
\end{center}

\vspace*{0.3cm}
\begin{center}
HASAN ARSLAN$^{*,1}$, ALNOUR ALTOUM$^{*,2}$,  MARIAM ZAAROUR$^{*,3}$ \\
$^{*}${\small {\textit{Department of  Mathematics, Faculty of Science, Erciyes University, 38039, Kayseri, Turkey}}}\\
{\small {\textit{ $^{1}$hasanarslan@erciyes.edu.tr}}}~~{\small {\textit{$^{2}$ alnouraltoum178@gmail.com}}}\\
{\small {\textit{$^{3}$mariamzaarour94@gmail.com }}}\\[0pt]
\end{center}

\begin{abstract}
In this paper, we construct a mixed-base number system over the generalized symmetric group $G(m,1,n)$, which is a complex reflection group with a root system of type $B_n^{(m)}$. We also establish one-to-one correspondence between all positive integers in the set $\{1,\cdots,m^nn!\}$ and the elements of $G(m,1,n)$ by constructing the subexceedant function in relation to this group. In addition, we provide a new enumeration system for $G(m,1,n)$ by defining the inversion statistic on $G(m,1,n)$. Finally, we prove that the \textit{flag-major index} is equi-distributed with this inversion statistic on $G(m,1,n)$. Therefore, the flag-major index is Mahonian on $G(m,1,n)$ with respect to the length function $L$.
\end{abstract}

\textbf{Keywords}: Weyl group, Complex reflection group, Subexceedant function, Integer representation, Root system.\\

\textbf{2020 Mathematics Subject Classification}: 20F55, 20C30. 
\\

\section{Introduction} 
The main purpose of this paper is to construct the integer representations of the generalized symmetric group $G(m,1,n)$. It is well-known that if $m=1$, then $G(1, 1, n)$ is the symmetric group $S_n$ which is a Weyl group of type $A_{n-1}$; for m=2, $G(2, 1, n)$ is the hyperoctahedral group which is a Weyl group of type $B_n$; and $G(2, 2, n)$ is a Weyl group of type $D_n$ (see \cite{br9}). In particular, the integer representations of the classical Weyl groups of types $A_{n-1},~B_n,~D_n$ were studied in \cite{br6},\cite{br8} and \cite{br1}, respectively. Throughout this paper, $\varepsilon$ denotes the $m$-th root of unity. For a fixed $m>1$, $G(m,1,n)$ denotes the generalized symmetric group which consists of all permutations $\pi$ acting on the set $I_n^{m}=\{\varepsilon^{k}i | i=1, \cdots, n; k=1, \cdots, m\}$. In other words, it consists of all bijections $\pi$ of the set $I_n^{m}$ onto itself providing that $\pi(\varepsilon^{k}i )=\varepsilon^{k}\pi(i)$ for all $i=1, \cdots, n; ~k=1, \cdots,m$. Given $x, y \in \mathbb{Z}$, $x \leq y$, we set $[x,y]:=\{x, x+1, \cdots, y\}$. Let $S_n$ be the symmetric group on $[1,n]$ and $C_m$ be the cyclic group of order $m$. Then $G(m,1,n)$ is a split group extension of $C_{m}^n$ by $S_n$, where $C_{m}^n$ is the direct product of $n$ copies of $C_m$ . The representation of $G(m,1,n)$ is given by the following form (see \cite{br4}):
\begin{align*}
    G(m,1,n)=&\langle s_1,\cdots, s_{n-1}, t_1,\cdots, t_{n} : s_i^2=(s_is_{i+1})^3=(s_is_j)^2=e, |i-j|>1;\\ & t_{i}^m=e, t_it_j=t_jt_i, s_it_is_i=t_{i+1}, s_it_j=t_js_i,~j\neq i,i+1 \rangle.
\end{align*}
The group $G(m,1,n)$ has a Cohen (Dynkin) diagram with respect to the set of generators $S=\{t_1,s_1,\cdots,s_{n-1}\}$ as follows:
\[
\begin{tikzpicture}[
roundednode/.style={circle, draw=white!60, fill=white!5, very thick, minimum size=6mm},
roundnode/.style={circle, draw=black!50, fill=white!5, very thick, minimum size=6mm},
squarednode/.style={circle, draw=black!60, fill=white!5, very thick, text width=0.005mm, minimum size=6mm},
]
\node[roundnode, font=\tiny]      (maintopic){m};
\node[roundednode]  [left=of maintopic]{$B_n^{(m)}$ :};
\node[roundnode]        (uppercircle)       [right=of maintopic]{};
\node[roundnode]      (rightsquare)       [right=of uppercircle]{};
\node[roundnode]        (lowercircle)       [right=of rightsquare]{};
\node[roundnode]        (zcircle)       [right=of lowercircle] {};
\node[node distance= 0.3cm, below = of maintopic]{$t_1$};
\node [label=$\frac{-1}{\sqrt{2}}$] (C) at (0.8,0) {};
\node[node distance= 0.3cm, below = of uppercircle]{$s_1$};
 \node[node distance= 0.3cm, below = of rightsquare]{$s_2$};
 \node[node distance= 0.3cm, below = of lowercircle]{$s_{n-2}$};
 \node[node distance= 0.3cm, below = of zcircle]{$s_{n-1}$};
\draw [-,double  distance=2pt,thick](maintopic) -- (uppercircle);
\draw[-] (uppercircle) -- (rightsquare);
\draw[densely dashed,-] (rightsquare)-- (lowercircle);
\draw[-] (lowercircle)-- (zcircle);
\draw[-](lowercircle)-- (zcircle);
\end{tikzpicture}\\
\]
The group $S_n$ is generated by $\{s_1,\cdots, s_{n-1}\}$, where $s_i~~(i=1,\cdots,n-1)$ is described with the transposition $(i,i+1)$. Here, $t_i~~(i=1,\cdots,n)$ may be defined as the permutation \[
  t_i = \bigl(\begin{smallmatrix}
    1 & 2 &  \cdots &i-1&~~i&~~i+1 \cdots & n \\
    1 & 2 & \cdots &i-1&~~ \varepsilon i&~~i+1 \cdots  &  n
  \end{smallmatrix}\bigr).
\]

From the above relations, we obtain that $\tau t_i \tau^{-1}=t_{\tau(i)}$ for any $\tau \in S_n$ and $i=1,\cdots, n$. Hence, we conclude for all $i=1, \cdots, n-1$ that $t_{i+1}=s_it_is_i$ is a reduced expression and the length of each $t_j~~(j=1,\cdots,n)$ equals to $2j-1$. In our work, we assume that the length function $\textit{l}$ on $G(m,1,n)$ is the function $G(m,1,n) \rightarrow \mathbb{N}_{0}$ associated with the set of generators $S$ in the sense of \cite{br2}. It is well-known that the cardinality of $G(m,1,n)$ is $m^nn!$. Therefore, any element $\pi$ of the group $G(m, 1, n)$ is written as follows (see \cite{br3}):
\[
  \pi = \bigl(\begin{smallmatrix}
    1 & 2 &  \cdots  & n \\
    \varepsilon^{r_1}\beta_1 & \varepsilon^{r_2}\beta_2 & \cdots   &  \varepsilon^{r_n}\beta_n
  \end{smallmatrix}\bigr)=\beta \prod_{k=1}^n t_{k}^{r_k}\in G(m,1,n),
\]
where $0\leq r_k \leq m-1$, ~$\beta = \bigl(\begin{smallmatrix}
    1 & 2 &  \cdots  & n \\
    \beta_1 & \beta_2 & \cdots   &  \beta_n
  \end{smallmatrix}\bigr) \in S_n$, and $ \beta_k= \beta(k)$ for all $k=1, \cdots,n$.
Let 
\[
  \sigma = \bigl(\begin{smallmatrix}
    1 & 2 &  \cdots  & n \\
    \varepsilon^{v_1}\gamma_1 & \varepsilon^{v_2}\gamma_2 & \cdots   &  \varepsilon^{v_n}\gamma_n
  \end{smallmatrix}\bigr) \in G(m,1,n),
\]
where $\gamma = \bigl(\begin{smallmatrix}
    1 & 2 &  \cdots  & n \\
    \gamma_1 & \gamma_2 & \cdots   &  \gamma_n
  \end{smallmatrix}\bigr) \in S_n$,~~
$0\leq v_i \leq m-1$ for all $i=1, \cdots,n$. Hence, based on \cite{br3} we obtain that 
\begin{align*}
  \pi \sigma &= \bigl(\begin{smallmatrix}
    1 & 2 &  \cdots  & n \\
    \varepsilon^{r_1}\beta_1 & \varepsilon^{r_2}\beta_2 & \cdots   &  \varepsilon^{r_n}\beta_n
  \end{smallmatrix}\bigr) \bigl(\begin{smallmatrix}
    1 & 2 &  \cdots  & n \\
    \varepsilon^{v_1}\gamma_1 & \varepsilon^{v_2}\gamma_2 & \cdots   &  \varepsilon^{v_n}\gamma_n
  \end{smallmatrix}\bigr)\\
  &=\bigl(\begin{smallmatrix}
    1 & 2 &  \cdots  & n \\
    \varepsilon^{v_1+r_{\gamma_1}}\beta_{\gamma_1} & \varepsilon^{v_2+r_{\gamma_2}}\beta_{\gamma_2}& \cdots   &  \varepsilon^{v_n+r_{\gamma_n}}\beta_{\gamma_n}
  \end{smallmatrix}\bigr) .
\end{align*}
As a convention, we assume throughout this paper that the rightmost permutation acts first in the multiplication of permutations.

The longest element $w_0=\prod_{k=1}^n t_{k}^{m-1}$ of the $G(m,1,n)$ can be written in the following form:
$$  w_0 = \bigl(\begin{smallmatrix}
    1 & 2 &  \cdots  & n \\
    \varepsilon^{m-1}~~ 1 ~~& \varepsilon^{m-1}~~ 2 ~~& ~~\cdots   ~~&  \varepsilon^{m-1}~~ n
  \end{smallmatrix}\bigr) \in G(m,1,n).
$$
Thus we may declare from \cite{br2} that the length of any reduced expression in $G(m,1,n)$ can not exceed $n(n+m-2)$, that is the length of $w_0$.

Following \cite{br0} we let, $\sigma_0:=t_1$ and for all $i \in [1,n-1]$,~$\sigma_i:=s_is_{i-1}\cdots s_1 t_1 \in G(m,1,n)$. Thus, the collection $\{\sigma_0, \sigma_1, \cdots, \sigma_{n-1}\}$ is a different set of generators for $G(m,1,n)$ and any $w\in G(m,1,n)$ has a unique expression 
\begin{equation}\label{flag0}
w=\sigma_{n-1}^{k_{n-1}}\cdots \sigma_{2}^{k_{2}}\sigma_{1}^{k_{1}}\sigma_{0}^{k_{0}}
\end{equation}
with $0\leq k_i \leq mi+m-1$ for all $0\leq i \leq n-1$.  Adin and Roichman defined a permutation statistic called as the \textit{flag-major index} for $G(m,1,n)$ in the case where $m\geq 2$ as follows (see \cite{br0}): Let $w\in G(m,1,n)$. Then
\begin{equation}\label{flag1}
fmaj(w)=\sum_{i=0}^{n-1}k_i.
\end{equation}
It is well-known from \cite{br5'} that 
\begin{equation}\label{flag2}
\sum_{w \in G(m,1,n)}q^{fmaj(w)}=\prod_{i=1}^n [im]_q
\end{equation}
where $q$ is an indeterminate and $[im]_q=\frac{1-q^{im}}{1-q}$ for every $i=1,\cdots, n$.

The rest of this paper is organised as follows: In the second section, we construct the $G_{m,n}$-type number system and give its combinatorial descriptions. In section 3, we establish subexceedant function for the generalized symmetric group $G(m,1,n)$ and show that there exists a one-to-one correspondence between any arbitrary positive integer in $\{1,\cdots,m^nn!\}$ and an element of this  group. In section 4, we study the inversion statistic on the generalized symmetric group. In addition, using the inversion notion we propose a new enumeration system for $G(m,1,n)$ and provide a new approach to the proof of Poincar\'e polynomial of $G(m,1,n)$. Finally, we will prove that the flag-major index and the inversion statistic are equi-distrubuted over $G(m,1,n)$ and that the flag-major index is Mahonian on $G(m,1,n)$ when considering the length function $L$ defined with respect to root system.

\section{Construction of $G_{m,n}$-type Number System}

Let $G(m,1,n)$ be a generalized symmetric group where $m$ and $n$ be two fixed positive integers. In this section, we introduce the $G_{m,n}$-type number system and give its some properties. 
To understand the $G_{m,n}$-type number system, we shall prove the next theorem.

\begin{thm}\label{exp}
Let $m$ be any fixed positive integer. Then every positive integer $x$ can always be expressed in the following form:
\begin{equation}\label{def}
  x=\sum_{i=0}^{n-1} d_iG_i
\end{equation}
where $0 \leq d_i \leq m(i+1)-1$ and $G_i=m^{i}i!$ for all the $0 \leq i \leq n-1$. 
\end{thm}

We will denote any positive integer $x$  of the form (\ref{def}) by
\[
x=(d_{n-1}:d_{n-2}: \cdots : d_1: d_0).
\]
\begin{defn}
The set of integers $x$, which can be written in the form (\ref{def}) based on a fixed positive integer $m$, is called the $G_{m,n}$\textit{-type number system}.
\end{defn}

To prove Theorem \ref{exp}, we need the following lemmas, which concern with some fundamental properties of the $G_{m,n}$-type number system. In particular, these properties have structures similar to those of the factoriadic number system of type $A_{n-1}$ and the hyperoctahedral base system of type $B_n$.

Two lemmas listed below generalize the results presented in\cite{br8} for hyperoctahedral groups.  

\begin{lem} \label{first}
For any $x=(d_{n-1}:d_{n-2}: \cdots : d_1: d_0)$ in the $G_{m,n}$-type number system, we have
\begin{equation*}
    0 \leq x \leq G_n-1.
\end{equation*}
\end{lem}

As a result of Lemma \ref{first}, we conclude that we have exactly $m^{n}n!$ numbers in the $G_{m,n}$-type number system.

\begin{lem}\label{second}
Let $x=(d_{n-1}:d_{n-2}: \cdots : d_1: d_0)$ be a number in $G_{m,n}$-type number system, then we have
\begin{equation*}
    d_{n-1}G_{n-1} \leq x < (d_{n-1}+1)G_{n-1}.
\end{equation*}
\end{lem}

\textit{$\boldsymbol{Proof~~of~~Theorem~~\ref{exp}}$}: \\
Assume that a positive integer $x$ has two representations in the $G_{m,n}$-type number system as follows:
\begin{equation*}
   x=(d_{n-1}:d_{n-2}: \cdots : d_1: d_0)=(e_{s-1}:e_{s-2}: \cdots : e_1: e_0),
\end{equation*}
where $d_{n-1}\neq 0$ and $e_{s-1}\neq 0$. The facts that both $d_{n-1}$ and $e_{s-1}$ are at least $1$ lead to
\begin{equation}\label{cntr}
   G_{n-1} \leq d_{n-1} G_{n-1} \leq x ~~~~
  \textrm{and}~~~~G_{s-1} \leq e_{s-1} G_{s-1} \leq x.
\end{equation}
Now, we suppose that $n \neq s$. Without loss of generality, we can assume that $n<s$. Then by Lemma \ref{first} and the inequality in the right hand side of the equation (\ref{cntr}), we obtain 
\begin{equation*}
    x < G_n \leq G_{s-1} \leq x,
\end{equation*}
which is a contradiction. Thus, we get $n=s$. 

By induction on the number of digits, we will show that $d_i=e_i$ for all $0 \leq i \leq n-1$. From the equation (\ref{def}), the assertion is clear for $x=(d_0)=(e_0)$. We assume that a positive integer $x$ with $k(<n)$ digits in the $G_{m,n}$-type number system has a unique representation. Suppose that $d_{n-1}\neq  e_{n-1}$. Without loss of generality, take $d_{n-1}< e_{n-1}$. Thus, from Lemma \ref{second} we get
\begin{equation*}
    x < (d_{n-1}+1) G_{n-1} \leq e_{n-1} G_{n-1} \leq x,
\end{equation*}
which is a contradiction and hence $d_{n-1}= e_{n-1}$. Since $d_{n-1}= e_{n-1}$ and by the induction hypothesis, the integer $x-d_{n-1}G_{n-1}=x-e_{n-1}G_{n-1}$ has a unique representation and so $d_i=e_i$ for all $0 \leq i \leq n-2$. This completes the proof.

Now, we will explain how to express any positive integer $x$ in terms of the $G_{m,n}$-type number system:\\
The algorithm proceeds in a series of steps. In the first step of the algorithm, $x$ is divided by $m$ and the reminder is set to be $r_0=d_0$ in the division process
\begin{equation*}
    x=m.q_0+r_0.
\end{equation*}
Then divide $q_0$ by $2m$ and the reminder is set to be $r_1=d_1$ in the division process
\begin{equation*}
    q_0=2.m.q_1+r_1.
\end{equation*}
Proceed in this way, i.e., by dividing $q_{i-1}$ by $m(i+1)$ and getting $r_i=d_i$ in the expression
\begin{equation*}
    q_{i-1}=(i+1).m.q_i+r_{i}
\end{equation*}
until the quotient $q_{n-1}=0$ is zero for some integer $n$. Thus, at the final step, we get 
\begin{equation*}
    q_{n-2}=n.m.q_{n-1}+r_{n-1}
\end{equation*}
and assign $r_{n-1}$ to $d_{n-1}$. Eventually, we write the number $x$ as 
\begin{equation}\label{intrep}
   x=(d_{n-1}:d_{n-2}: \cdots : d_1: d_0)
\end{equation}
in $G_{m,n}$-type base system.

\begin{rem}
Have initially selected any fixed positive integer $m$ and then applied the above procedure, we can write any positive integer $x$ in the $G_{m,n}$-type number system as in the equation (\ref{intrep}).
 Here, $n$ is the number of digits that we obtain when we apply the above procedure to the positive integer $x$. Hence, $n$ is revealed at the end of the operations.

\end{rem}

We propose to use the following Python algorithm to convert any positive integer into a number in $G_{m,n}$-type base system:\\
\textbf{Algorithm 1:} \hspace*{\fill} \\
x=\textrm{int(input}('\textrm{Enter}~~\textrm{any}~~\textrm{positive}~~ \textrm{integer:}'))\\
m=\textrm{int(input}('\textrm{Enter}~~\textrm{the}~~\textrm{value of}~~ \textrm{m:}'))\\
\textrm{for}~~\textrm{i}~~\textrm{in}~~\textrm{range}(1,x):\\
$d=x \% (m*i)$\\
if $x > 0:$\\
$x=x//(m*i)$\\
else:\\
break\\
print(d, end=':') \\

Let us give an example of how this algorithm works:  

\begin{ex}
The expression of positive integer $x=199761$ in $G_{7,5}$-type base system is $x=(3:13:1:5:2)$.
\end{ex}

Alternatively, the following Python algorithm can be used to convert any number in the $G_{m,n}$-type number system into a positive integer:\\

\textbf{Algorithm 2:}  \hspace*{\fill} \\
n=\textrm{int(input}('\textrm{enter}~~\textrm{the}~~\textrm{index}~~ \textrm{of}~~\textrm{$G_{m,n}$}~~ \textrm{base}:'))\\
m=\textrm{int(input}('\textrm{enter}~~\textrm{the}~~\textrm{value}~~ \textrm{of}~~\textrm{$m$}:'))\\
f=1\\
x=0\\
\textrm{for}~~\textrm{i}~~\textrm{in}~~\textrm{range}(0,n):\\
d=\textrm{int(input}('\textrm{enter}~~\textrm{a}~~\textrm{number}~~ \textrm{in}~~$G_{m,n}$\textrm{-type}~~ \textrm{number}~~\textrm{system}:'))\\
if $i==0$ or $i==1$:\\
$f=1$\\
else: \\
$f=f*i$\\
$t=m**(i)*f$\\
$z=d*t$\\
$x~ +=z$\\
print('The decimal number is:', x) \\

\begin{ex}
Let $x=(56:238:8:270:218:133:236:210:204:102:63:208:157:94:171:89:19:20:50:67:121:134:75:30:37:58:97:104:58:2:75:31:42:24:43:2:17:3:16:16:0:3)$ be a number in $G_{7,42}$-type number system. It corresponds to the positive integer $847269328185736775326682798778327079883274857780833279866982328472693276\\65908932687971$.
\end{ex}

\section{Integer Representation of the group $G(m,1,n)$}

Mantaci and Rakotondrajao \cite{br7} introduced the subexceedant functions for the symmetric group $S_n$ and showed that there is one-to-one correspondence between permutations in $S_n$ and the subexceedant functions. The subexceedant function is an important tool for constructing integer representations of the classical Weyl groups, (see \cite{br5,br8,br1}).  Depending more on the structure of the group $G(m,1,n)$ and inspiring by\cite{br3}, \cite{br7} and \cite{br8}, we will define the subexceedant function for $G(m,1,n)$.

\begin{defn}[\cite{br7}]
A subexceedant function $f$ on the set $[1,n]$ is a map such that 
\begin{equation*}
    1 \leq f(i) \leq i,~~\textrm{for}~~\textrm{all}~~1 \leq i \leq n.
\end{equation*}
\end{defn}
We denote the set of all subexceedant functions on $[1,n]$ by $\mathcal{F}_n$ and hence $|\mathcal{F}_n|=n!$. The subexceedant function $f$ on $[1,n]$ is, in general, identified with the word $f(1);\cdots; f(n)$. Furthermore, the map 
\begin{equation}\label{subex}
   \psi : \mathcal{F}_n \mapsto S_n,~~\psi(f)=(nf(n))\cdots(2f(2))(1f(1)) 
\end{equation}
is a bijection and $(if(i))$ is a transposition for each $i=1,\cdots, n$ \cite{br7}.

Now, let $\beta = \bigl(\begin{smallmatrix}
    1 & 2 &  \cdots  & n \\
    \beta_1 & \beta_2 & \cdots   &  \beta_n
\end{smallmatrix}\bigr)$ be an element of $S_n$. In \cite{br7}, Mantaci and Rakotondrajao defined the subexceedant function $f$ corresponding to $\beta$ according to the map $\psi$ with the following steps:
 \begin{itemize}
     \item Set $f(n)=\beta_n$.
     \item Then, replace the image of ${\beta}^{-1}(n)$ in the permutation $\beta = \bigl(\begin{smallmatrix}
    1 & 2 &  \cdots  & n \\
    \beta_1 & \beta_2 & \cdots   &  \beta_n
\end{smallmatrix}\bigr)$ by $\beta_n$. Thus, a new permutation $\beta'$ that contains $n$ as a fixed point is obtained. Hence, $\beta'$ can be considered as an element of $S_{n-1}$.
     \item Set $f(n-1)=\beta'_{n-1}$.
     \item Apply the same procedure for the permutation $\beta'$, i.e, exchange the image of ${\beta'}^{-1}(n-1)$ in the permutation $\beta'$ with $\beta'_{n-1}$, and determine in this manner $f(n-2)$.
     \item Continue this iteration until you find all the $f(i)$ values for each $1 \leq i \leq n$.
 \end{itemize}

Now, we are ready to define the subexceedant function for the generalized symmetric group $G(m,1,n)$.

\begin{defn}
Let $x=(d_{n-1}:d_{n-2}: \cdots : d_1: d_0)$ be a number with the $n$-digits in the $G_{m,n}$-type number system. We define the subexceedant function $f$, associated with $x$, on the set $[1,n]$ as follows: 
\begin{equation}\label{subexc}
 f(i)=1+\lfloor{\frac{d_{i-1}}{m}}\rfloor,~~\textrm{for}~~\textrm{all}~~1 \leq i \leq n
\end{equation}
where $\lfloor\cdot \rfloor$ denotes the floor function.
\end{defn}
It is clear that $1 \leq f(i) \leq i$ since $0 \leq d_{i-1} \leq m . i-1$ for all $1 \leq i \leq n$.
Having defined the coefficient $c_i=\varepsilon^{d_{i-1}}$ for all $1 \leq i \leq n$, we associate $x=(d_{n-1}:d_{n-2}: \cdots : d_1: d_0)$ in the $G_{m,n}$-type number system to a unique permutation 
\[
  \pi_x = \bigl(\begin{smallmatrix}
    1 & 2 &  \cdots  & n \\
    c_1\beta_1 & c_2\beta_2 & \cdots   &  c_n \beta_n
  \end{smallmatrix}\bigr) \in G(m,1,n),
\]
where $\beta_f = \bigl(\begin{smallmatrix}
    1 & 2 &  \cdots  & n \\
    \beta_1 & \beta_2 & \cdots   &  \beta_n
  \end{smallmatrix}\bigr) \in S_n$
is the permutation of $S_n$, that is the image $\psi(f)$ of the subexceedant function $f$ under $\psi$ given in the equation (\ref{subex}).

As a result of the above facts, we match each positive integer $x$ expressed in the $G_{m,n}$-type number system with an element of $G(m,1,n)$. On the other hand, we will now show how any element of this group is converted to a positive integer. For this reason, we take any permutation 
$$    \sigma = \bigl(\begin{smallmatrix}
    1 & 2 &  \cdots  & n \\
    \varepsilon^{r_1}\gamma_1 & \varepsilon^{r_2}\gamma_2 & \cdots   &  \varepsilon^{r_n}\gamma_n
  \end{smallmatrix}\bigr) \in G(m,1,n),
$$
where $\gamma = \bigl(\begin{smallmatrix}
    1 & 2 &  \cdots  & n \\
    \gamma_1 & \gamma_2 & \cdots   &  \gamma_n
  \end{smallmatrix}\bigr) \in S_n$,~~
$0\leq r_i \leq m-1$. 
\begin{enumerate}
\item First of all, we determine the subexceedent function $f$ in relation to $\gamma$.  Let $f=\psi^{-1}(\gamma) \in \mathcal{F}_n$.
 \item Set $d_i=m(f(i+1)-1)+r_{i+1},~ \textrm{for}~~\textrm{all}~~0 \leq i \leq n-1$.
 \item Compute $x=(d_{n-1}:d_{n-2}: \cdots : d_1: d_0)$.
\end{enumerate}
Based on the above facts, we will state the following theorem without proof.

\begin{thm}\label{funda}
There is one-to-one correspondence between natural integers of the set $\{1,\cdots,m^n n!\}$ and elements of the generalized symmetric group $G(m,1,n)$.
\end{thm}
Throughout this work, we will denote by $I(w)$ the integer representation of an element $w\in G(m,1,n)$. 

\begin{ex}
The corresponding integer representation of $x=2161$ in the $G_{3,5}$-type number system is $(1:1:3:0:1)$. We obtain the subexceedant function of $x$ based on the equation (\ref{subexc}) as $f=f(1);f(2);f(3);f(4);f(5)=1;1;2;1;1$. Hence, we get $\pi_x= \bigl(\begin{smallmatrix}
    1 &~~ 2 &~~ 3 &~~ 4 &~~ 5  \\
    \varepsilon 3 &~~ 4 &~~2 &~~
    \varepsilon5 &~~
    \varepsilon1
  \end{smallmatrix}\bigr)  \in G(3,1,5)$, where $\gamma = \bigl(\begin{smallmatrix}
    1 & 2 & 3 & 4  & 5 \\
    3 & 4 & 2 & 5 &  1
  \end{smallmatrix}\bigr) \in S_5$.
\end{ex}

\begin{ex}
Let $\sigma = \bigl(\begin{smallmatrix}
    1 & ~~2 &~~ 3 &~~ 4 &~~ 5 &~~ 6 \\
    \varepsilon^{2}2 &~~ \varepsilon^{3}4 &~~ \varepsilon 3 &~~1&~~  \varepsilon^{2}6 &~~  \varepsilon 5
  \end{smallmatrix}\bigr) \in G(4,1,6),$
where $\gamma = \bigl(\begin{smallmatrix}
    1 & 2 & 3 & 4 & 5 & 6 \\
    2 & 4 & 3 & 1 & 6 & 5
  \end{smallmatrix}\bigr) \in S_6.$
We obtain the subexceedant function associated with $\gamma$ as 
$$f=f(1);f(2);f(3);f(4);f(5);f(6)=1;1;3;1;5;5.$$ 
Thus, the corresponding integer representation of $\sigma$ is 
$$I(\sigma)=(13:14:0:7:3:2)$$
and its positive integer value is $406433$.
\end{ex}

Moreover, we deduce that the subexceedant function $f$ associating with the longest element $w_0$ of $G(m,1,n)$ is $f (1); f(2); \cdots ; f (n)=1;2; \cdots ; n$. 

\begin{cor}\label{corollary}
Let $w_0$ be the longest element of the generalized symmetric group $G(m,1,n)$. Then the integer representation of $w_0$ in the $G_{m,n}$-type number system has the following form: 
\begin{align*}
I(w_0)=&(d_{n-1}:d_{n-2}: \cdots : d_2: d_1: d_0)\\
=&(nm-1: (n-1)m-1:\cdots:3m-1:2m-1:m-1).
\end{align*}
Therefore, it is clear that the order of group $G(m,1,n)$ is 
$$|G(m,1,n)|=\prod_{i=0}^{n-1}(d_i+1)=m^{n}n!.$$

\end{cor}

The exponents of the group $G(m,1,n)$ are well-known as $e_1=m-1,~e_2=2m-1,\cdots, e_n=nm-1$ from \cite{br12}. Notice that all components in the integer representation of $w_0$ in Corollary \ref{corollary} are actually nothing else but the exponents of the group $G(m,1,n)$. 

The following algorithm is useful for finding the subexceedant function corresponding to any number $d$ in the $G_{m,n}$-type number system:\\

\textbf{Algorithm 3:} \hspace*{\fill} \\
from math import floor\\
n=int(input('Enter the value of n:'))\\
m=int(input('Enter the value of m:'))\\
c=[]\\
for i in range(1,n+1):\\
    d=int(input('Enter digit in $G_(m,n)$ base system:'))\\
    $f=1+\textrm{floor}(d/m)$\\
    c.append(f)\\
print("The values of subexceedant function is:", c)\\

\begin{ex}
Let $d=(17:14:11:8:5:2)\in G_{3,6}$. The corresponding subexceedant function of $d$ is immediately obtained as $f (1); f(2); f(3); f(4); f(5); f(6)=1;2;3;4;5;6$ when using Algorithm 3.
\end{ex}

The following Python algorithm is also used to convert any text message into its numerical value:\\

\textbf{Algorithm 4:} \hspace*{\fill} \\
print("Enter a string: ", end="")\\
text = input()\\
for char in text:\\
    ascii = ord(char)\\
    print(ascii, end="")\\

\begin{ex}
Taking the sentence THE QUICK BROWN FOX JUMPS OVER THE LAZY DOG, then its ASCII code is $[84, 72, 69, 32, 81, 85, 73, 67, 75, 32, 66, 82, \\ 79, 87, 78, 32, 70, 79, 88, 32, 74, 85, 77, 80, 83, 32, 79, 86, 69, 82, 32, 84, 72,  69, 32, 76, 65,\\ 90, 89, 32, 68, 79, 71].$
If we get m=7, then we obtain the integer representation of $84726932818573677532668279877832707988327485778083327986698232847269327\\665908932687971$ in $G_{7,42}$-type number system as  
$(56:238:8:270:218:133:236:210:204:102:63:208:157:94:171:89:19:20:50:67:121:134:75:30:37:58:97:104:58:2:75:31:42:24:43:2:17:3:16:16:0:3)$ using Algorithm 1.
\end{ex}

\section{Inversion Statistic on $G(m,1,n)$}
A finite Weyl group has two canonical length functions, essentially identical, which reveal a great deal of facts about the structure of the group. The first length function is defined as the length of the reduced expressions in terms of the set of standard generators, while the other length function is defined by the effect of the Weyl group on the positive root system. For the generalized symmetric group, the length function $l$ defined by reduced expressions with respect to the set of the canonical generators $S=\{t_1, s_1, \cdots, s_{n-1}\}$, is unfortunately not a very useful tool for studying the structure of such groups. For this reason, Bremke and Malle in \cite{br2} defined a new length function based on the generalized root system, which resembles the root system of $G(m,1,n)$ in \cite{br3'}.

Let $V$ be a complex vector space $\mathbb{C}^n$ with the standard unitary inner product. Let $\{e_1,\cdots, e_n\}$ be the set of standard basis vectors of $V$. Actually, the imprimitive complex reflection group $G(m,1,n)\subset GL_n(\mathbb{C})$ is generated by the reflections $s_1,\cdots,s_{n-1}$ of order $2$ associated with the roots $e_2-e_1,\cdots, e_n-e_{n-1}$, respectively, and a (complex) reflection $t_1$ of order $m$ with root $e_1$. Our next aim is to describe a new statistic on $G(m,1,n)$. To introduce this new statistic on $G(m,1,n)$, it would make more sense to think of $G(m,1,n)$ as a complex reflection group with the following root system given in \cite{br2}:
\[
\Phi=\{\varepsilon^{i}e_j-\varepsilon^{k}e_l  ~~|~~\varepsilon^{i}e_j\neq \varepsilon^{k}e_l,~~0\leq i, k\leq m-1,~~1 \leq j, l\leq n\}.\\
\]
Positive and negative root systems are defined in the following way, respectively:
\begin{align*}
\Phi^{+}=&\{\varepsilon^{i}e_j-\varepsilon^{k}e_j \in \Phi~~|~~0\leq i<k\leq m-1,~~1 \leq j \leq n\}\\
&\cup \{e_j-\varepsilon^{k}e_l \in \Phi~~|~~0\leq k\leq m-1,~~1 \leq l < j \leq n\}\\
&\cup \{\varepsilon^{i}e_j-\varepsilon^{k}e_l \in \Phi~~|~~0 \leq i, k\leq m-1,~~k\neq 0,~~1 \leq j < l \leq n\},
\end{align*}
and $\Phi^{-}=\Phi\setminus\Phi^{+}=-\Phi^{+}$. It is clear that $|\Phi|=mn(mn-1)$ and $|\Phi^{+}|=|\Phi^{-}|=\frac{|\Phi|}{2}$. Now, let 
\begin{equation}\label{simplesystem}
\Delta =\{ e_j-\varepsilon^{k}e_l  \in \Phi~~|~~0\leq k\leq m-1,~~1 \leq l\leq j \leq n \} \subset \Phi^{+}.
\end{equation}
Note that, $\Delta$ is a subset of $\Phi^{+}$ and it contains only one root for each reflection in $G(m,1,n)$. Therefore, a triple $(\Phi,\Phi^{-},\Delta)$ is so-called a \textit{root system} of type $B_n^{(m)}$ for the complex reflection group $G(m,1,n)$ (see \cite{br2}). From \cite{br2}, the length function $\textit{L}$ connected with the root system $(\Phi,\Phi^{-},\Delta)$ is defined as 
\begin{equation}\label{len}
\textit{L}~:~G(m,1,n) \rightarrow \mathbb{N}_0,~~~\textit{L}(w)=|w(\Delta) \cap \Phi^{-}|.
\end{equation}
Moreover, one can easily check the fact that  $$w_0(\Delta)=\varepsilon^{m-1}\Delta \subset \Phi^{-}$$ holds, where $w_0 = \bigl(\begin{smallmatrix}
    1 & 2 &  \cdots  & n \\
    \varepsilon^{m-1}~~ 1 ~~~~& \varepsilon^{m-1}~~ 2 ~~&~~~~ \cdots   ~~~~&  \varepsilon^{m-1}~~ n
  \end{smallmatrix}\bigr) \in G(m,1,n)$
is the longest element ($\textit{l}(w_0)=n(n+m-2)$) of $G(m,1,n)$  in relation to the length function $\textit{l}$, which is defined by reduced expressions with respect to the set of the canonical generators $S$.
\begin{rem}
It should be noted here that the two length functions $l$ and $L$ defined on $G(m,1,n)$ are generally not identical with each other. It is also clear that in the case of $m=2$ the length function $L$ coincides with the canonical length function of the Weyl group of type $C_n$.
\end{rem}

\begin{defn}
Let $\sigma \in S_n$. A pair $(i,j) \in [1,n] \times [1,n]$ is called an inversion of $\sigma$ if $i<j$ and $\sigma_i>\sigma_j$ \cite{br10}. 
\end{defn}

Now, we define 
\begin{align*}
\Delta_i=&\{e_{n+1-i}-\varepsilon^{k}e_{n+1-i} \in \Delta~:~0 < k \leq m-1\}\\
&\cup \{e_{n+1-i}-\varepsilon^{k}e_j \in \Delta~:~0 \leq k \leq m-1,~~j<n+1-i \leq n\}.
\end{align*}
for all $i=1,\cdots, n$. Furthermore, the length of the longest element $w_0$ of $G(m,1,n)$ (with respect to the definition in (\ref{len})) is equal to $\textit{L}(w_0)=|\Delta|=n(m-1)+\frac{nm(n-1)}{2}$. Clearly, $|\Delta_i|=m(n-i+1)-1$ for all $i=1,\cdots, n$.
\begin{defn}
For any $w \in G(m,1,n)$, we define the number of $i$-inversions of $w$ for all $i=1,\cdots, n$ as follows:
\[
inv_i(w)=|w(\Delta_i) \cap \Phi^{-}|.
\]

\end{defn}

Based on the definitions of the concepts $\Delta_i$ and $inv_i(w)$, we can conclude the following theorem.
\begin{thm}\label{22}
We have the following combinatorial properties:
\begin{enumerate}
    \item The set $\Delta$ given in (\ref{simplesystem}) can be decomposed as
\begin{equation*}
    \Delta=\bigsqcup_{i=1}^n \Delta_i.
\end{equation*}
    \item The length of any $w \in G(m,1,n)$ is expressed as $\textit{L}(w)=\sum_{i=1}^n inv_i (w)$.
    
    \item For any $w \in G(m,1,n)$ and for all $i=1,\cdots, n$, we have  
\begin{equation}\label{newexp}
        0 \leq inv_i (w) \leq m(n-i+1)-1.
\end{equation}
\end{enumerate}
\end{thm}  
As a consequence of part (2) of Theorem \ref{22}, we deduce that the sum of the $i$-inversions of any given element $w\in G(m,1,n)$, will be denoted by $inv(w)$, can be actually used to calculate the length of the element depending on the length function $\textit{L}$.
Denote the inversion sequence $(inv_1(w):\cdots: inv_n(w))$ of $i$-inversions of $w$ by $Inv(w)$. The sequence $Inv(w)=(inv_1(w):\cdots: inv_n(w))$ is called the  \textit{inversion table} of $w$. As a result of the equation (\ref{newexp}), we say that the inversion table possesses the same properties as a number in the $G_{m,n}$-type number system. Therefore, this enable us to introduce a new kind of classification of all elements of $G(m,1,n)$. When we enumerate all the elements of the group $G(m,1,n)$ in lexicographic order, we assign each element of the group to a different integer in the following manner: Given the inversion table $Inv(w)=(inv_1(w):\cdots: inv_n(w))$ of $w$, then we define the \textit{rank} of $w$ as $x+1$, where $x$ is the positive integer corresponding to the number $(inv_1(w):\cdots: inv_n(w))$ in the $G_{m,n}$-type number system. This means that a new enumeration system is created on $G(m,1,n)$.

\begin{ex}
Let $(\Phi,\Phi^{-},\Delta)$ be the root system of type $B_3^{(3)}$, where
\begin{align*}
\Delta=&\{e_3-\varepsilon e_3, e_3-\varepsilon^2 e_3, e_3-e_1, e_3-\varepsilon e_1, e_3-\varepsilon^2 e_1, e_3-e_2, e_3-\varepsilon e_2, e_3-\varepsilon^2 e_2\}\\
&\cup \{e_2-\varepsilon e_2, e_2-\varepsilon^2 e_2, e_2-e_1, e_2-\varepsilon e_1, e_2-\varepsilon^2 e_1\}\\
&\cup \{e_1-\varepsilon e_1, e_1-\varepsilon^2 e_1\}
\end{align*}
and 
\begin{align*}
\Phi^{-}=&\{\varepsilon^{k}e_j -\varepsilon^{i}e_j\in \Phi~~|~~0\leq i<k\leq 2,~~1 \leq j \leq 3\}\\
&\cup \{\varepsilon^{k}e_l-e_j \in \Phi~~|~~0\leq k\leq 2,~~1 \leq l < j \leq 3\}\\
&\cup \{\varepsilon^{k}e_l-\varepsilon^{i}e_j \in \Phi~~|~~0\leq i, k\leq 2,~~k\neq 0,~~1 \leq j < l \leq 3\}.
\end{align*}
We determine $\Delta_1=\{e_3-\varepsilon e_3, e_3-\varepsilon^2 e_3, e_3-e_1, e_3-\varepsilon e_1, e_3-\varepsilon^2 e_1, e_3-e_2, e_3-\varepsilon e_2, e_3-\varepsilon^2 e_2\}$,~~~$\Delta_2= \{e_2-\varepsilon e_2, e_2-\varepsilon^2 e_2, e_2-e_1, e_2-\varepsilon e_1, e_2-\varepsilon^2 e_1\}$  and $\Delta_3=\{e_1-\varepsilon e_1, e_1-\varepsilon^2 e_1\}$. Since $w_0(\Delta)=\varepsilon^2(\Delta)\subset \Phi^{-}$ for $w_0 \in G(3,1,3)$, then we obtain the inversion table of $w_0$ as $Inv(w_0)=(inv_1(w_0):inv_2(w_0):inv_3(w_0))=(8:5:2)$. Thus the rank of $w_0$ is $162$, which is the order of the group $G(3,1,3)$. If we take another element $w = \bigl(\begin{smallmatrix}
    1 & ~~2 & ~~3 \\
    \varepsilon^{2} 3 &~~ \varepsilon 1   & ~~ 2
  \end{smallmatrix}\bigr)$ of $G(3,1,3)$, then we find the inversion table of $w$ as $Inv(w)=(inv_1(w):inv_2(w):inv_3(w))=(1:2:2)$, where $w(\Delta_1) \cap \Phi^{-}=\{ e_2-e_3\},~~       w(\Delta_2) \cap \Phi^{-}=\{\varepsilon e_1-e_1,~\varepsilon e_1-e_3\}$ and $w(\Delta_3) \cap \Phi^{-}=\{\varepsilon^2 e_3-e_3,~\varepsilon^2 e_3-\varepsilon e_3\}$.   The rank of $w$ is $27$.
\end{ex}

With the help of the next theorem, one can practically derive the inversion table $Inv(w)=(inv_1(w):\cdots: inv_n(w))$ of a  $w\in G(m,1,n)$   without using the root system structure. 

\begin{thm}\label{3}
For $w=\beta \prod_{k=1}^n t_{k}^{r_k} \in G(m,1,n)$, we have 
\begin{equation}\label{33}
inv_i(w)=r_{n+1-i}+m.|\{(j,n+1-i) : j<n+1-i, \beta_j<\beta_{n+1-i}, r_{n+1-i} \neq 0\}|+inv_{i}(\beta)
\end{equation}
for all $i=1,\cdots,n$, where $inv_i(\beta)=|\{(j,n+1-i) : j<n+1-i, \beta_j>\beta_{n+1-i} \}|$ in $S_n$. More precisely, we have $inv_i(w)=inv_i(\beta)$ when $ r_{n+1-i}=0$.
\end{thm}

\begin{proof}
First of all, we decompose the set $\Delta_i$ for each $i=1,\cdots,n$ into two parts $\Delta_i=\Delta_{i,1} \bigsqcup \Delta_{i,2}$ such that
\[
\Delta_{i,1}=\{e_{n+1-i}-\varepsilon^{k}e_{n+1-i} \in \Delta~:~0 < k \leq m-1\}~~ \textrm{and} 
\]
\[
\Delta_{i,2}= \{e_{n+1-i}-\varepsilon^{k}e_j \in \Delta~:~0 \leq k \leq m-1,~~j<n+1-i \leq n\}.
\]
Then $w(e_{n+1-i}-\varepsilon^{k}e_{n+1-i})=\varepsilon^{r_{n+1-i}}e_{\beta_{n+1-i}}-\varepsilon^{k+r_{n+1-i}}e_{\beta_{n+1-i}} \in\Phi^{-}$ if and only if $k+r_{n+1-i}\geq m$. Thus, the number of roots in $\Delta_{i,1}$ turned into a root in $\Phi^{-}$ by $w$ equals to $r_{n+1-i}$.\\
For $e_{n+1-i}-\varepsilon^{k}e_j$, we have $w(e_{n+1-i}-\varepsilon^{k}e_j)=\varepsilon^{r_{n+1-i}}e_{\beta_{n+1-i}}-\varepsilon^{k+r_j}e_{\beta_j}$, which lies in  $\Phi^{-}$ if and only if either $\beta_j< \beta_{n+1-i}$ and $r_{n+1-i} \neq 0$ (where $k$ precisely takes $m$ values) or  $\beta_j> \beta_{n+1-i}$ and $
k+r_j=m$. Therefore, we get the desired formula as follows:
\begin{align*}
inv_i(w)=&r_{n+1-i}+m.|\{(j,n+1-i) : j<n+1-i, \beta_j<\beta_{n+1-i}, r_{n+1-i} \neq 0\}|\\
&+|\{(j,n+1-i) : j<n+1-i, \beta_j>\beta_{n+1-i} \}|\\
=&r_{n+1-i}+m.|\{(j,n+1-i) : j<n+1-i, \beta_j<\beta_{n+1-i}, r_{n+1-i} \neq 0\}|\\
&+inv_{i}(\beta).
\end{align*}
\end{proof}

\begin{ex}
	Let $w= \bigl(\begin{smallmatrix}
		1 & ~~2 & ~~ 3&~~ 4&~~5&~~6 \\
		\varepsilon^{2}3&~~ \varepsilon^{4}1 &~~ \varepsilon6 &~~5&~~ \varepsilon 4&~~  \varepsilon^{2}2
	\end{smallmatrix}\bigr)\in G(5,1,6).$ Taking into account the equation (\ref{33}) we obtain the inversion table of $w$ as $Inv(w)=(11:13:1:11:5:2)$, and so we conclude that the length of $w$ is $L(w)=43$ and that rank of $w$ is $4321328$ using Algorithm 2.
\end{ex}

\begin{ex}
	In the Table \ref{2}, one can respectively see all ranks, 1-inversions, 2-inversions and 3-inversions of the 162 elements of $G(3,1,3)$ using Theorem \ref{3}, where $\varepsilon$ is $e^{\frac{2i \pi}{3}}$. In the following table, we will write any permutation $w$ in $G(3,1,3)$ in one-line notation as $w_1w_2w_3$.
\newpage
\renewcommand{\arraystretch}{0.0 }
	\begin{table}[h!]
	\caption{Inversion table of the group $G(3,1,3)$}\label{2}
		\label{tab:table1}		
		\centering
			\scalebox{0.9}{
			\begin{tabular}{|m{0.8cm}|m{0.3cm}m{0.4cm}m{0.4cm}|m{1.1cm}||m{0.8cm}|m{0.3cm}m{0.4cm}m{0.4cm}|m{1.1cm}||m{0.8cm}|m{0.3cm}m{0.4cm}m{0.4cm}|m{1.1cm}|}
				\hline
				\centering
				Rank&\textbf{$w_1 w_2 w_3$}&&& $\textrm{Inv}(w)$&Rank&\textbf{$w_1 w_2 w_3$}&&& $\textrm{Inv}(w)$&Rank&\textbf{$w_1 w_2 w_3$}&&&$\textrm{Inv}(w)$\\
				\hline\hline
				$\mathbf{1}$&1 &2 & 3 & 0:0:0 &$\mathbf{55}$&2 & 3 & $\varepsilon$1 & 3:0:0 &$\mathbf{109}$&1 & 3 & $\varepsilon^2$2 & 6:0:0\\
				\hline
				$\mathbf{2}$&$\varepsilon$1&  2&  3 & 0:0:1 &$\mathbf{56}$& $\varepsilon$2 &3& $\varepsilon$1 & 3:0:1 &$\mathbf{110}$& $\varepsilon$1& 3 &$\varepsilon^2$2 & 6:0:1\\
				\hline
				$\mathbf{3}$&$\varepsilon^2$1&  2 & 3 & 0:0:2 &$\mathbf{57}$& $\varepsilon^2$2& 3& $\varepsilon$1 & 3:0:2 &$\mathbf{111}$& $\varepsilon^2$1 &3 &$\varepsilon^2$2 & 6:0:2\\
				\hline
				$\mathbf{4}$&2&  1&  3 & 0:1:0 &$\mathbf{58}$& 3 &2& $\varepsilon$1 & 3:1:0 &$\mathbf{112}$&3 &1 &$\varepsilon^2$2 & 6:1:0\\
				\hline
				$\mathbf{5}$&$\varepsilon$2&  1 & 3 & 0:1:1 &$\mathbf{59}$& $\varepsilon$3& 2 &$\varepsilon$1 & 3:1:1 &$\mathbf{113}$& $\varepsilon$3& 1 &$\varepsilon^2$2 & 6:1:1\\
				\hline
				$\mathbf{6}$&$\varepsilon^2$2&  1 & 3 & 0:1:2 &$\mathbf{60}$&$\varepsilon^2$3& 2 &$\varepsilon$1 & 3:1:2 &$\mathbf{114}$ &$\varepsilon^2$3& 1 &$\varepsilon^2$2 & 6:1:2\\
				\hline
				$\mathbf{7}$&2 & $\varepsilon$1&  3 & 0:2:0 &$\mathbf{61}$& 3 &$\varepsilon$2& $\varepsilon$1 & 3:2:0 &$\mathbf{115}$& 3& $\varepsilon$1& $\varepsilon^2$2 & 6:2:0\\
				\hline
				$\mathbf{8}$&$\varepsilon$2&  $\varepsilon$1&  3 & 0:2:1 &$\mathbf{62}$&$\varepsilon$3& $\varepsilon$2& $\varepsilon$1 & 3:2:1 &$\mathbf{116}$&$\varepsilon$3& $\varepsilon$1& $\varepsilon^2$2 & 6:2:1\\
				\hline
				$\mathbf{9}$&$\varepsilon^2$2&  $\varepsilon$1&  3 & 0:2:2 &$\mathbf{63}$ &$\varepsilon^2$3& $\varepsilon$2& $\varepsilon$1 & 3:2:2 &$\mathbf{117}$&$\varepsilon^2$3& $\varepsilon$1& $\varepsilon^2$2 & 6:2:2\\
				\hline
				$\mathbf{10}$&2  &$\varepsilon^2$1&  3 & 0:3:0 &$\mathbf{64}$& 3 &$\varepsilon^2$2& $\varepsilon$1 & 3:3:0 &$\mathbf{118}$ &3& $\varepsilon^2$1& $\varepsilon^2$2 & 6:3:0\\
				\hline
				$\mathbf{11}$&$\varepsilon$2 & $\varepsilon^2$1 & 3 & 0:3:1 &$\mathbf{65}$& $\varepsilon$3& $\varepsilon^2$2 &$\varepsilon$1 & 3:3:1 &$\mathbf{119}$& $\varepsilon$3 &$\varepsilon^2$1& $\varepsilon^2$2 & 6:3:1\\
				\hline
				$\mathbf{12}$&$\varepsilon^2$2&  $\varepsilon^2$1&  3 & 0:3:2 &$\mathbf{66}$& $\varepsilon^2$3& $\varepsilon^2$2 &$\varepsilon$1 & 3:3:2 &$\mathbf{120}$& $\varepsilon^2$3& $\varepsilon^2$1& $\varepsilon^2$2 & 6:3:2\\
				\hline
				$\mathbf{13}$&1& $\varepsilon$2& 3& 0:4:0 &$\mathbf{67}$& 2 &$\varepsilon$3& $\varepsilon$1 & 3:4:0 &$\mathbf{121}$&1 &$\varepsilon$3& $\varepsilon^2$2 & 6:4:0\\
				\hline
				$\mathbf{14}$&$\varepsilon$1 &$\varepsilon$2& 3 & 0:4:1 &$\mathbf{68}$& $\varepsilon$2& $\varepsilon$3& $\varepsilon$1 & 3:4:1 &$\mathbf{122}$& $\varepsilon$1& $\varepsilon$3& $\varepsilon^2$2 & 6:4:1\\
				\hline
				$\mathbf{15}$&$\varepsilon^2$1& $\varepsilon$2& 3 & 0:4:2 &$\mathbf{69}$ &$\varepsilon^2$2& $\varepsilon$3& $\varepsilon$1 & 3:4:2 &$\mathbf{123}$&$\varepsilon^2$1 &$\varepsilon$3& $\varepsilon^2$2 & 6:4:2\\
				\hline
				$\mathbf{16}$&1& $\varepsilon^2$2& 3& 0:5:0 &$\mathbf{70}$& 2 &$\varepsilon^2$3& $\varepsilon$1 & 3:5:0 &$\mathbf{124}$ &1& $\varepsilon^2$3& $\varepsilon^2$2 & 6:5:0\\
				\hline
				$\mathbf{17}$&$\varepsilon$1& $\varepsilon^2$2& 3& 0:5:1 &$\mathbf{71}$& $\varepsilon$2& $\varepsilon^2$3& $\varepsilon$1 & 3:5:1 &$\mathbf{125}$ &$\varepsilon$1& $\varepsilon^2$3& $\varepsilon^2$2 & 6:5:1\\
				\hline
				$\mathbf{18}$&$\varepsilon^2$1& $\varepsilon^2$2& 3& 0:5:2 &$\mathbf{72}$& $\varepsilon^2$2& $\varepsilon^2$3& $\varepsilon$1 & 3:5:2 &$\mathbf{126}$ &$\varepsilon^2$1 &$\varepsilon^2$3& $\varepsilon^2$2 & 6:5:2\\
				\hline
				$\mathbf{19}$&1& 3& 2& 1:0:0 &$\mathbf{73}$& 2& 3& $\varepsilon^2$1 & 4:0:0 &$\mathbf{127}$&1& 2& $\varepsilon$3 & 7:0:0\\
				\hline
				$\mathbf{20}$& $\varepsilon$1& 3& 2& 1:0:1 &$\mathbf{74}$& $\varepsilon$2& 3 & $\varepsilon^2$1 & 4:0:1 &$\mathbf{128}$ &$\varepsilon$1& 2 &$\varepsilon$3 & 7:0:1\\
				\hline
				$\mathbf{21}$& $\varepsilon^2$1& 3& 2& 1:0:2 &$\mathbf{75}$&  $\varepsilon^2$2& 3 & $\varepsilon^2$1 & 4:0:2 &$\mathbf{129}$&$\varepsilon^2$1& 2 &$\varepsilon$3 & 7:0:2\\
				\hline
				$\mathbf{22}$&3& 1& 2& 1:1:0 &$\mathbf{76}$& 3&  2&  $\varepsilon^2$1 & 4:1:0 &$\mathbf{130}$ &2& 1& $\varepsilon$3 & 7:1:0\\
				\hline
				$\mathbf{23}$&$\varepsilon$3& 1& 2& 1:1:1 &$\mathbf{77}$& $\varepsilon$3& 2&  $\varepsilon^2$1 & 4:1:1 &$\mathbf{131}$ &$\varepsilon$2& 1 &$\varepsilon$3 & 7:1:1\\
				\hline
				$\mathbf{24}$&$\varepsilon^2$3& 1& 2& 1:1:2 &$\mathbf{78}$&  $\varepsilon^2$3& 2&  $\varepsilon^2$1 &4:1:2 &$\mathbf{132}$ &$\varepsilon^2$2& 1 &$\varepsilon$3 & 7:1:2\\
				\hline
				$\mathbf{25}$&3& $\varepsilon$1& 2& 1:2:0 &$\mathbf{79}$& 3&  $\varepsilon$2&  $\varepsilon^2$1 & 4:2:0 &$\mathbf{133}$ &2 &$\varepsilon$1 &$\varepsilon$3 & 7:2:0\\
				\hline	
				$\mathbf{26}$&$\varepsilon$3 &$\varepsilon$1 &2& 1:2:1 &$\mathbf{80}$&  $\varepsilon$3&  $\varepsilon$2  &$\varepsilon^2$1 & 4:2:1 &$\mathbf{134}$&$\varepsilon$2& $\varepsilon$1& $\varepsilon$3 & 7:2:1\\
				\hline
				$\mathbf{27}$&$\varepsilon^2$3 &$\varepsilon$1& 2& 1:2:2 &$\mathbf{81}$& $\varepsilon^2$3 & $\varepsilon$2 & $\varepsilon^2$1 & 4:2:2 &$\mathbf{135}$ &$\varepsilon^2$2& $\varepsilon$1& $\varepsilon$3 & 7:2:2\\
				\hline
				$\mathbf{28}$&3 &$\varepsilon^2$1 &2& 1:3:0 &$\mathbf{82}$&3 & $\varepsilon^2$2&  $\varepsilon^2$1 & 4:3:0 &$\mathbf{136}$ &2 &$\varepsilon^2$1 &$\varepsilon$3 & 7:3:0\\
				\hline
				$\mathbf{29}$&$\varepsilon$3& $\varepsilon^2$1& 2& 1:3:1 &$\mathbf{83}$& $\varepsilon$3&  $\varepsilon^2$2 & $\varepsilon^2$1 & 4:3:1 &$\mathbf{137}$&$\varepsilon$2 &$\varepsilon^2$1& $\varepsilon$3 & 7:3:1\\
				\hline
				$\mathbf{30}$&$\varepsilon^2$3 &$\varepsilon^2$1& 2& 1:3:2 &$\mathbf{84}$& $\varepsilon^2$3 & $\varepsilon^2$2&  $\varepsilon^2$1 & 4:3:2 &$\mathbf{138}$ &$\varepsilon^2$2& $\varepsilon^2$1 &$\varepsilon$3 & 7:3:2\\
				\hline
				$\mathbf{31}$&1& $\varepsilon$3& 2& 1:4:0 &$\mathbf{85}$&2  &$\varepsilon$3 & $\varepsilon^2$1 & 4:4:0 &$\mathbf{139}$&1& $\varepsilon$2 &$\varepsilon$3 & 7:4:0\\
				\hline
				$\mathbf{32}$&$\varepsilon$1& $\varepsilon$3 &2& 1:4:1 &$\mathbf{86}$&  $\varepsilon$2 & $\varepsilon$3&  $\varepsilon^2$1 & 4:4:1 &$\mathbf{140}$&$\varepsilon$1& $\varepsilon$2& $\varepsilon$3 & 7:4:1\\
				\hline
				$\mathbf{33}$&$\varepsilon^2$1& $\varepsilon$3 &2& 1:4:2 &$\mathbf{87}$&  $\varepsilon^2$2&  $\varepsilon$3 & $\varepsilon^2$1 & 4:4:2 &$\mathbf{141}$&$\varepsilon^2$1 &$\varepsilon$2& $\varepsilon$3 & 7:4:2\\
				\hline
				$\mathbf{34}$&1 &$\varepsilon^2$3 &2& 1:5:0 &$\mathbf{88}$&2 & $\varepsilon^2$3 & $\varepsilon^2$1 & 4:5:0 &$\mathbf{142}$&1& $\varepsilon^2$2 &$\varepsilon$3 & 7:5:0\\
				\hline
				$\mathbf{35}$&$\varepsilon$1& $\varepsilon^2$3 &2& 1:5:1 &$\mathbf{89}$&  $\varepsilon$2&  $\varepsilon^2$3 & $\varepsilon^2$1 & 4:5:1 &$\mathbf{143}$&$\varepsilon$1 &$\varepsilon^2$2& $\varepsilon$3 & 7:5:1\\
				\hline
				$\mathbf{36}$&$\varepsilon^2$1& $\varepsilon^2$3 &2& 1:5:2 &$\mathbf{90}$&  $\varepsilon^2$2&  $\varepsilon^2$3 & $\varepsilon^2$1 & 4:5:2 &$\mathbf{144}$ &$\varepsilon^2$1& $\varepsilon^2$2& $\varepsilon$3 & 7:5:2\\
				\hline
				$\mathbf{37}$&2 &3 &1& 2:0:0 &$\mathbf{91}$& 1& 3 & $\varepsilon$2 & 5:0:0 &$\mathbf{145}$ &1 &2& $\varepsilon^2$3 & 8:0:0\\
				\hline
				$\mathbf{38}$&$\varepsilon$2 &3& 1& 2:0:1 &$\mathbf{92}$&$\varepsilon$1& 3 &$\varepsilon$2 & 5:0:1 &$\mathbf{146}$ &$\varepsilon$1& 2 &$\varepsilon^2$3 & 8:0:1\\
				\hline
				$\mathbf{39}$&$\varepsilon^2$2& 3& 1& 2:0:2 &$\mathbf{93}$& $\varepsilon^2$1& 3 &$\varepsilon$2 & 5:0:2 &$\mathbf{147}$ &$\varepsilon^2$1& 2 &$\varepsilon^2$3 & 8:0:2\\
				\hline
				$\mathbf{40}$&3& 2 &1& 2:1:0 &$\mathbf{94}$&3 &1 &$\varepsilon$2 & 5:1:0 &$\mathbf{148}$&2 &1& $\varepsilon^2$3 & 8:1:0\\
				\hline
				$\mathbf{41}$&$\varepsilon$3 &2 &1& 2:1:1 &$\mathbf{95}$& $\varepsilon$3 &1 &$\varepsilon$2 & 5:1:1 &$\mathbf{149}$ &$\varepsilon$2 &1 &$\varepsilon^2$3 & 8:1:1\\
				\hline
				$\mathbf{42}$&$\varepsilon^2$3& 2& 1& 2:1:2 &$\mathbf{96}$&$\varepsilon^2$3 &1& $\varepsilon$2 & 5:1:2 &$\mathbf{150}$ &$\varepsilon^2$2& 1& $\varepsilon^2$3 & 8:1:2\\
				\hline
				$\mathbf{43}$&3 &$\varepsilon$2 &1& 2:2:0 &$\mathbf{97}$&3 &$\varepsilon$1& $\varepsilon$2 & 5:2:0 &$\mathbf{151}$ &2 &$\varepsilon$1& $\varepsilon^2$3 & 8:2:0\\
				\hline
				$\mathbf{44}$&$\varepsilon$3 &$\varepsilon$2& 1& 2:2:1 &$\mathbf{98}$&$\varepsilon$3& $\varepsilon$1& $\varepsilon$2 & 5:2:1 &$\mathbf{152}$ &$\varepsilon$2& $\varepsilon$1& $\varepsilon^2$3 & 8:2:1\\
				\hline
				$\mathbf{45}$&$\varepsilon^2$3 &$\varepsilon$2& 1& 2:2:2 &$\mathbf{99}$&$\varepsilon^2$3 &$\varepsilon$1& $\varepsilon$2 & 5:2:2 &$\mathbf{153}$ &$\varepsilon^2$2& $\varepsilon$1& $\varepsilon^2$3 & 8:2:2\\
				\hline
				$\mathbf{46}$&3 &$\varepsilon^2$2& 1& 2:3:0 &$\mathbf{100}$& 3 &$\varepsilon^2$1& $\varepsilon$2 & 5:3:0 &$\mathbf{154}$&2 &$\varepsilon^2$1 &$\varepsilon^2$3 & 8:3:0\\
				\hline
				$\mathbf{47}$&$\varepsilon$3 &$\varepsilon^2$2 &1& 2:3:1 &$\mathbf{101}$&$\varepsilon$3 &$\varepsilon^2$1 &$\varepsilon$2 & 5:3:1 &$\mathbf{155}$ &$\varepsilon$2 &$\varepsilon^2$1& $\varepsilon^2$3 & 8:3:1\\
				\hline
				$\mathbf{48}$&$\varepsilon^2$3& $\varepsilon^2$2& 1& 2:3:2 &$\mathbf{102}$& $\varepsilon^2$3& $\varepsilon^2$1 &$\varepsilon$2 & 5:3:2 &$\mathbf{156}$&$\varepsilon^2$2 &$\varepsilon^2$1& $\varepsilon^2$3 & 8:3:2\\
				\hline
				$\mathbf{49}$&2 &$\varepsilon$3& 1& 2:4:0 &$\mathbf{103}$&1 &$\varepsilon$3 &$\varepsilon$2 & 5:4:0 &$\mathbf{157}$&1 &$\varepsilon$2& $\varepsilon^2$3 & 8:4:0\\
				\hline
				$\mathbf{50}$&$\varepsilon$2 &$\varepsilon$3 &1& 2:4:1 &$\mathbf{104}$& $\varepsilon$1& $\varepsilon$3& $\varepsilon$2 & 5:4:1 &$\mathbf{158}$ &$\varepsilon$1 &$\varepsilon$2 &$\varepsilon^2$3 & 8:4:1\\
				\hline
				$\mathbf{51}$&$\varepsilon^2$2& $\varepsilon$3& 1& 2:4:2 &$\mathbf{105}$&$\varepsilon^2$1& $\varepsilon$3 &$\varepsilon$2 & 5:4:2 &$\mathbf{159}$ &$\varepsilon^2$1& $\varepsilon$2& $\varepsilon^2$3 & 8:4:2\\
				\hline
				$\mathbf{52}$&2 &$\varepsilon^2$3& 1& 2:5:0 &$\mathbf{106}$&1& $\varepsilon^2$3& $\varepsilon$2 & 5:5:0 &$\mathbf{160}$ &1 &$\varepsilon^2$2& $\varepsilon^2$3 & 8:5:0\\
				\hline
				$\mathbf{53}$&$\varepsilon$2& $\varepsilon^2$3& 1& 2:5:1 &$\mathbf{107}$&$\varepsilon$1& $\varepsilon^2$3& $\varepsilon$2 & 5:5:1 &$\mathbf{161}$ &$\varepsilon$1& $\varepsilon^2$2& $\varepsilon^2$3 & 8:5:1\\
				\hline
				$\mathbf{54}$&$\varepsilon^2$2& $\varepsilon^2$3& 1& 2:5:2 &$\mathbf{108}$&$\varepsilon^2$1& $\varepsilon^2$3& $\varepsilon$2 & 5:5:2 &$\mathbf{162}$ &$\varepsilon^2$1& $\varepsilon^2$2& $\varepsilon^2$3 & 8:5:2\\
				\hline
			\end{tabular}}
		\end{table}
\end{ex}

Now, on the contrary, we will give an effective method of how to associate a permutation $w\in G(m,1,n)$ with a given rank value as follows:
For this reason, firstly fix the positive integer $m$. Then consider a rank $k$. After that, convert $k-1$ to a number in $G_{m,n}$-type number system. As is well-known, each number in the $G_{m,n}$-type number system gives us an unique permutation in the generalized symmetric group $G(m,1,n)$. Let us represent $k-1$ by $(d_{n-1}:\cdots :d_1:d_0)\in G_{m,n}$. In fact, we are going to generate a $w$ permutation here such that
\[
inv_i(w)=d_{n-i}~\textrm{for}~\textrm{all}~i=1,\cdots,n.
\]

We will establish a $w$ permutation corresponding to $(d_{n-1}:\cdots :d_1:d_0)$ by proceeding the following steps:
 \begin{itemize}
     \item After listing all possible values that the desired permutation can take as   
\begin{equation}\label{converse1}
n,\cdots,1;\varepsilon1,\cdots, \varepsilon^{m-1}1;\cdots;\varepsilon n,\cdots, \varepsilon^{m-1}n
\end{equation}
number them from $0$ to $nm-1$ by starting with the leftmost value. Then determine the value corresponding to the number $d_{n-1}$ from (\ref{converse1}) and assign it to $w_n$.\\
     \item Say $w_n:=\varepsilon^{r_i}i$. Remove the terms $i,\varepsilon i,\cdots, \varepsilon^{m-1}i$  from (\ref{converse1}). Reorder the remaining values as 
\begin{equation}\label{converse2}
\begin{split}
n,\cdots,i+1,i-1,\cdots,1;&\varepsilon1,\cdots, \varepsilon^{m-1}1;\cdots;\varepsilon (i-1),\cdots, \varepsilon^{m-1}(i-1);\\
&\varepsilon (i+1),\cdots, \varepsilon^{m-1}(i+1);\cdots;\varepsilon n,\cdots, \varepsilon^{m-1}n
\end{split}
\end{equation}
and renumber them from $0$ to $(n-1)m-1$ by starting with the leftmost term. Then find the value corresponding to the number $d_{n-2}$ from (\ref{converse2}) and set it as $w_{n-1}$.
     \item Apply the same procedure to (\ref{converse2}) and determine in this manner $w_{n-2}$.
     \item Continue these operations until you determine all $w_i$ values for each $1 \leq i \leq n$.
 \end{itemize}

Let us consider the following example to illustrate this method.
\begin{ex}
We will find the $4321328^{th}$ group element of $G(5,1,6)$. If we apply Algorithm 1 to the positive integer $4321327$, then we obtain the inversion table of the group element which we are looking for as $Inv(w)=(11:13:1:11:5:2)$.  \\

\scalebox{0.7}{
	\begin{tabular}{|m{1.2cm}|m{1.1cm}|m{1.9cm}|m{0.2cm}m{0.3cm}m{0.3cm}m{0.4cm}|m{0.2cm}m{0.3cm}m{0.3cm}m{0.4cm}|m{0.2cm}m{0.3cm}m{0.3cm}m{0.4cm}|m{0.2cm}m{0.3cm}m{0.3cm}m{0.4cm}|m{0.2cm}m{0.3cm}m{0.3cm}m{0.4cm}|m{0.2cm}m{0.3cm}m{0.3cm}m{0.4cm}|}
		\hline
		\multirow{2}{*}{$1^{th}$ step} &P.P.V.&6 5 4 3 2 1 & $\varepsilon$1& $\varepsilon^2$1&$\varepsilon^3$1& $\varepsilon^4$1&$\varepsilon$2& \color{blue} $\varepsilon^2$2&$\varepsilon^3$2&$\varepsilon^4$2& $\varepsilon$3&$\varepsilon^2$3&$\varepsilon^3$3&$\varepsilon^4$3&$\varepsilon$4&$\varepsilon^2$4&$\varepsilon^3$4&$\varepsilon^4$4&$\varepsilon$5&$\varepsilon^2$5&$\varepsilon^3$5&$\varepsilon^4$5&$\varepsilon$6&$\varepsilon^2$6&$\varepsilon^3$6&$\varepsilon^4$6\\
		\cline{2-27}
		&P.I.V.&0 1 2 3 4 5& 6 &7 &8& 9&10& \color{blue} 11&12&13&14&15&16&17&18&19&20&21&22&23&24&25&26&27&28&29\\
		
		\hline
\end{tabular}}

\scalebox{0.7}{
	\begin{tabular}{|m{1.2cm}|m{1.1cm}|m{1.9cm}|m{0.2cm}m{0.3cm}m{0.3cm}m{0.4cm}|m{0.2cm}m{0.3cm}m{0.3cm}m{0.4cm}|m{0.2cm}m{0.3cm}m{0.3cm}m{0.4cm}|m{0.2cm}m{0.3cm}m{0.3cm}m{0.4cm}|m{0.2cm}m{0.3cm}m{0.3cm}m{0.4cm}|}
		\hline
		\multirow{2}{*}{$2^{th}$ step} &P.P.V.&6 5 4 3  1 & $\varepsilon$1& $\varepsilon^2$1&$\varepsilon^3$1& $\varepsilon^4$1& $\varepsilon$3&$\varepsilon^2$3&$\varepsilon^3$3&$\varepsilon^4$3& \color{blue} $\varepsilon$4&$\varepsilon^2$4&$\varepsilon^3$4&$\varepsilon^4$4&$\varepsilon$5&$\varepsilon^2$5&$\varepsilon^3$5&$\varepsilon^4$5&$\varepsilon$6&$\varepsilon^2$6&$\varepsilon^3$6&$\varepsilon^4$6\\
		\cline{2-23}
		&P.I.V.&0 1 2 3 4 &5& 6 &7 &8& 9&10&11&12&\color{blue} 13&14&15&16&17&18&19&20&21&22&23&24\\
		
		\hline
\end{tabular}}

\scalebox{0.7}{
	\begin{tabular}{|m{1.2cm}|m{1.1cm}|m{1.9cm}|m{0.2cm}m{0.3cm}m{0.3cm}m{0.4cm}|m{0.2cm}m{0.3cm}m{0.3cm}m{0.4cm}|m{0.2cm}m{0.3cm}m{0.3cm}m{0.4cm}|m{0.2cm}m{0.3cm}m{0.3cm}m{0.4cm}|}
		\hline
		\multirow{2}{*}{$3^{th}$ step}&P.P.V.&6 \color{blue} 5 \color{black} 3 1& $\varepsilon$1& $\varepsilon^2$1&$\varepsilon^3$1& $\varepsilon^4$1& $\varepsilon$3&$\varepsilon^2$3&$\varepsilon^3$3&$\varepsilon^4$3&$\varepsilon$5&$\varepsilon^2$5&$\varepsilon^3$5&$\varepsilon^4$5&$\varepsilon$6&$\varepsilon^2$6&$\varepsilon^3$6&$\varepsilon^4$6\\
		\cline{2-19}
		&P.I.V.&0 \color{blue} 1 \color{black} 2 3&4& 5& 6 &7 &8& 9&10&11&12&13&14&15&16&17&18&19\\
		\hline
\end{tabular}}

\scalebox{0.7}{
	\begin{tabular}{|m{1.2cm}|m{1.1cm}|m{1.9cm}|m{0.2cm}m{0.3cm}m{0.3cm}m{0.4cm}|m{0.2cm}m{0.3cm}m{0.3cm}m{0.4cm}|m{0.2cm}m{0.3cm}m{0.3cm}m{0.4cm}|}
		\hline
		\multirow{2}{*}{$4^{th}$ step}&P.P.V.& 6 3 1 & $\varepsilon$1& $\varepsilon^2$1&$\varepsilon^3$1& $\varepsilon^4$1& $\varepsilon$3&$\varepsilon^2$3&$\varepsilon^3$3&$\varepsilon^4$3& \color{blue} $\varepsilon$6&$\varepsilon^2$6&$\varepsilon^3$6&$\varepsilon^4$6\\
		\cline{2-15}
		&P.I.V.&0 1  2& 3&4& 5& 6 &7 &8& 9&10&\color{blue} 11&12&13&14\\
		\hline
\end{tabular}}

\scalebox{0.7}{
	\begin{tabular}{|m{1.2cm}|m{1.1cm}|m{1.9cm}|m{0.2cm}m{0.3cm}m{0.3cm}m{0.4cm}|m{0.2cm}m{0.3cm}m{0.3cm}m{0.4cm}|}
		\hline
		\multirow{2}{*}{$5^{th}$ step} &P.P.V.& 3 1& $\varepsilon$1& $\varepsilon^2$1&$\varepsilon^3$1& \color{blue}$\varepsilon^4$1& $\varepsilon$3&$\varepsilon^2$3&$\varepsilon^3$3&$\varepsilon^4$3\\
		\cline{2-11}
		&P.I.V.&0 1 & 2& 3&4& \color{blue}5& 6 &7 &8& 9\\
		\hline
\end{tabular}}

\scalebox{0.7}{
	\begin{tabular}{|m{1.2cm}|m{1.1cm}|m{1.9cm}|m{0.2cm}m{0.3cm}m{0.3cm}m{0.4cm}|}
		\hline
		\multirow{2}{*}{$6^{th}$ step}  &P.P.V.& 3&  $\varepsilon$3&\color{blue}$\varepsilon^2$3&$\varepsilon^3$3&$\varepsilon^4$3\\
		\cline{2-7}
		&P.I.V.&0 &1&\color{blue}2& 3&4\\
		\hline
\end{tabular}}\\

In the above table, P.P.V. and P.I.V. stand for the abbreviations of possible values that the desired permutation can take and possible inversion values, respectively. Using the table, then $w\in G(5,1,6)$ is built up as 
\begin{equation*}
 w = \bigl(\begin{smallmatrix}
    1 & ~2 &~ 3&~4&~5 & ~6 \\
    \varepsilon^{2}3 &~ \varepsilon^{4}1 &~\varepsilon 6  & ~ 5&~\varepsilon 4&~\varepsilon^{2}2
  \end{smallmatrix}\bigr).
\end{equation*}

\end{ex}

The above procedures for constructing $w\in G(m,1,n)$ from its inversion table $Inv(w)$ establishes the following result. Thus we may state the following result without proof.
\begin{prop}\label{Tmn} 
	Let
	\[
	\mathcal{T}_{m,n}=\{(a_1 : \cdots : a_n) ~~|~~ 0 \leq a_i \leq m(n-i+1)-1\}=[0,nm-1]\times [0,(n-1)m-1]\times \cdots \times [0,m-1]. 
	\] 
	The map $I : G(m,1,n) \rightarrow \mathcal{T}_{m,n}$ that sends each permutation to its inversion table is a bijection. 
\end{prop}
Therefore, the inverson table $Inv(w)$ is another way to represent a permutation $w$. The fact that  $\textit{L}(w)=\sum_{i=1}^n inv_i (w)$ for any $w \in G(m,1,n)$ gives us a new approach to the proof of Poincar\'e polynomial for $G(m,1,n)$ in the sense of $S_n$ (see \cite{br10}).

\begin{cor}\label{stanley}
Let $inv(w)$ denote the sum of $i$-inversions of the permutation $w \in G(m,1,n)$. Then
\begin{equation*}
\sum_{w \in G(m,1,n)}q^{inv(w)}=\prod_{i=1}^n [im]_q
\end{equation*}
where $q$ is an indeterminate and $[im]_q=\frac{1-q^{im}}{1-q}$ for every $i=1,\cdots, n$.
\end{cor}

\begin{proof}
There always exists some $(a_1:\cdots: a_n)\in \mathcal{T}_{m,n}$ such that $L(w)=\sum_{i=1}^na_i$ for each $w\in G(m,1,n)$, for if we have at least $Inv(w) \in \mathcal{T}_{m,n}$ since $inv(w)=\textit{L}(w)=\sum_{i=1}^n inv_i (w)$. Considering also Proposition \ref{Tmn}, we have
\begin{align*}
\sum_{w \in G(m,1,n)}q^{inv(w)}=\sum_{w \in G(m,1,n)}q^{\textit{L}(w)}=&\sum_{a_1=0}^{nm-1}~~\sum_{a_2=0}^{(n-1)m-1}\cdots \sum_{a_n=0}^{m-1}q^{a_1+a_2+\cdots+a_n}\\
=&(\sum_{a_1=0}^{nm-1}q^{a_1}) (\sum_{a_2=0}^{(n-1)m-1}q^{a_2}) \cdots (\sum_{a_n=0}^{m-1}q^{a_n})\\
=&\prod_{i=1}^n [im]_q
\end{align*}
as desired.
\end{proof}

For any $w\in G(m,1,n)$ with the flag major index $fmaj(w)=k_0+k_1+\cdots+k_{n-1}$, where $0\leq k_i \leq mi+m-1$,~$i=0,1,\cdots,n-1$. Then we have
\begin{align*}
\sum_{w \in G(m,1,n)}q^{fmaj(w)}=&\sum_{k_{n-1}=0}^{nm-1}~~\sum_{k_{n-2}=0}^{(n-1)m-1}\cdots \sum_{k_0=0}^{m-1}q^{k_0+k_1+\cdots+k_{n-1}}\\
=&(\sum_{k_{n-1}=0}^{nm-1}q^{k_{n-1}}) (\sum_{k_{n-2}=0}^{(n-1)m-1}q^{k_{n-2}}) \cdots (\sum_{k_0=0}^{m-1}q^{k_0})\\
=&\prod_{i=1}^n [im]_q.
\end{align*}

Considering Corollary \ref{stanley} and the above result together, we conclude that the flag-major index and the inversion statistic are equi-distrubuted over $G(m,1,n)$. Furthermore, we can define a map $\phi: G(m,1,n) \rightarrow G(m,1,n)$  such that $inv(w)=fmaj(\phi(w))$ for each $w \in G(m,1,n)$ as follows: Let the inversion table of $w$ be $Inv(w)=(a_{n-1}:\cdots:a_1:a_0)$. If we define
$$\phi(w)=\sigma_{n-1}^{a_{n-1}}\cdots \sigma_{1}^{a_{1}}\sigma_{0}^{a_{0}}$$
then it is obvious that $\phi$ is a bijection and $inv(w)=\sum_{i=0}^{n-1} a_i=fmaj(\phi(w)),$  where $\sigma_{i}$ is the same as in (\ref{flag0}) for each $i,~0 \leq i \leq n-1$. As a result of these facts, we are ready to give a new interpretation of the flag-major index, which is a well-known parameter. 

\begin{thm}
The flag-major index is a Mahonian statistic on $G(m,1,n)$ with respect to the length function $L$.   
\end{thm}

\section{Conclusion}
In this paper, we define a mixed-radix number system over the generalized symmetric group $G(m,1,n)$. In addition, we introduce a one-to-one correspondence between the positive integers of the set $\{1,...,m^{n}n!\}$ and the elements of the generalized symmetric group $G(m,1,n)$. Then we use this concept as an efficient tool to provide a new enumeration system for $G(m,1,n)$. With this integer representation, one can use the elements of this group more effectively in cryptography. In other words, the integer representation for $G(m,1,n)$ may allow to build robust cryptographic algorithms based on the structure of this group due to using of two special parameters (these are $m$ and $n$). For the applications in the cryptography of classical Weyl groups, see \cite{br5, br8}. It is a natural question to ask whether such integer representations also exist for the other complex reflection groups such as $G(m,m,n)$ and $G(m,p,n)$, where $p$ is prime such that $p|m$ and $G(m,m,n) \subset G(m,p,n)\subset G(m,1,n)$. As a matter of fact, $i$-inversion statistics of Weyl group $A_{n-1}$ and $B_n$ were studied in \cite{br6} and \cite{br11}, respectively. We introduce and study $i$-inversion concept for any element of $G(m,1,n)$. Then we use it as an efficient tool to create a new enumeration system for $G(m,1,n)$. 

It is a natural question to ask here when we consider two permutation statistics-inversion number and the flag-major index, then how can \textit{Haglund-Remmel-Wilson identity} (in the case of symmetric group, the proof was given by Remmel and Wilson in \cite{br9'}) be defined for $G(m,1,n)$?

\end{document}